\documentclass{article}
\usepackage{amssymb}
\usepackage{hyperref}
\usepackage{amsfonts}
\usepackage{theorem}
\usepackage{graphicx}
\usepackage{amsmath}
\usepackage{amsfonts}
\usepackage{latexsym}
\usepackage{mathrsfs}

\newtheorem{theorem}{Theorem}

\newtheorem{lemma}[theorem]{Lemma}

\newtheorem{problem}[theorem]{Problem}
\newtheorem{proposition}[theorem]{Proposition}

\newenvironment{proof}{\paragraph{Proof:}}{\hfill$\square$}
\DeclareMathAlphabet{\ams}{U}{msb}{m}{n}

\input{boxedeps} 
\SetEPSFDirectory{} 
\HideDisplacementBoxes
\SetRokickiEPSFSpecial  %% dvips by Tom Rokicki

\begin{document}

\title{ Answer To Graham Higman's Question Regarding Januarials}
\author{Qaiser Mushtaq and Saadia Mehwish }

\begin{center}
{\LARGE Januarials of simple and general type}

Saadia Mehwish$^{1}$ and Qaiser Mushtaq$^{2}$
\end{center}

$^{1}${\footnotesize Department\ of\ Mathematics,}${\footnotesize \ }$%
{\footnotesize The Islamia University of Bahawalpur, Bahawalpur 63100,
Pakistan.}

\ {\footnotesize E-mail: saadiamehwish@gmail.com.}

$^{2}${\footnotesize The Islamia University of Bahawalpur, Bahawalpur 63100,
Pakistan.}

\ {\footnotesize E-mail: pir\_qmushtaq@yahoo.com.}\ 

\begin{center}
\bigskip \textbf{Abstract}
\end{center}

{\small Januarials were defined by Graham Higman in his last series of
lectures. In this paper we answer some questions posed by Higman in these
lectures.}

\begin{center}
{\small \bigskip }
\end{center}

{\tiny AMS Classification[2010]: Primary 05C25; Secondary 20G40}


{\tiny \bigskip }

\section{Introduction}

\bigskip Coset diagrams were introduced by Graham Higman as a diagram
representing an action of a particular group on a set or a space. 
% B. Everitt 
% \cite{br1} used these diagrams to prove that all but finitely many
% alternating groups are Hurwitz groups.
Several papers, namely \cite{br1, br, qmf, ww, at} appeared in which these
coset diagrams were used to prove interesting topological and group
theoretical results. Q. Mushtaq \cite{Qm} laid the foundations of coset
diagrams for the action of the modular group on finite Galois fields and
real quadratic irrational number fields. % A coset
% diagram for the modular group $\left\langle x,y:x^{2}=y^{3}=1\right\rangle $
% acting on a finite field is defined to be a graph whose vertices are the
% points of the field and edges define the action of the generators $x$ and $y$%
% . If $x$ interchanges the points $u$ and $v$, we use an unoriented edge
% connecting $u$ and $v$ to represent this. Three edges in a cycle formed by
% the action of $y$ form a triangle with the convention that it be directed
% anti-clockwise.

Graham Higman introduced the concept of januarials in his lectures at Oxford
in 2000, which turned out to be his last -- see \cite{jan}. In these he
posed several questions. This paper answers two of them -- see Proposition %
\ref{p} and Theorem \ref{t} below. % In \cite{jan} januarial of simple

\section{Januarials}

\label{januarials}

Let $\Delta=\Delta(2,k,\ell)$ be the abstract group with presentation 
\begin{equation*}
\Delta (2,k,\ell )=\left\langle x,y:x^{2}=y^{k}=(xy)^{\ell}=1\right\rangle
\end{equation*}
where $k,\ell\in\ams{Z}^{\geq 2}\cup\infty$. Such groups are called triangle
groups as they act discretely on the $2$-sphere, Euclidean plane or
hyperbolic plane with fundamental region a triangle. The $\Delta(2,k,\infty)$
are called \emph{Hecke groups\/}; in particular when $k=3$ we obtain the
classical modular group.

Suppose that we have an action of $\Delta $ on some set $S$. A \emph{coset
graph\/} $\Gamma $ is a directed graph in the plane depicting the action as
follows: the vertices of the graph are the elements of the set $S$; if two
vertices are transposed by the action of $x$ then there is an undirected $x$%
-edge connecting them; vertices fixed by $x$ have no $x$-edge incident with
them; there is a directed $y$-edge from $u$ to $v$ when $u$ is sent to $v$
under the action of $y$.

Now associate to each vertex of the coset graph a permutation of the
vertices incident to it as in Figure \ref{figure:123}.

\begin{figure}
  \centering
\BoxedEPSF{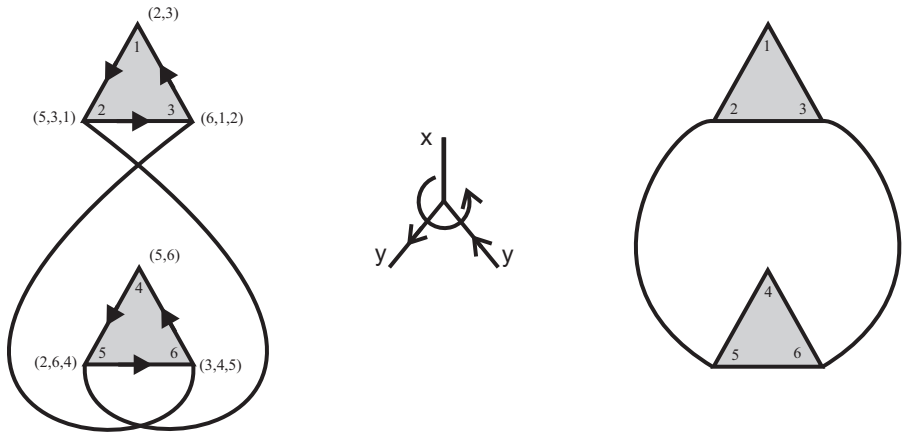 scaled 750}
  \caption{$2$-cell
embedding using the method described by \protect\cite[Theorem
6.47]{white84}}
\label{figure:123}
\end{figure}

% \begin{equation*}
% \FRAME{itbpFU}{4.4555in}{2.1378in}{0in}{\Qcb{Figure 1: $2$\mathit{-cell
% embedding using the method described by \protect\cite[Theorem 6.47]{white84}}%
% }}{}{Figure}{\special{language "Scientific Word";type
% "GRAPHIC";maintain-aspect-ratio TRUE;display "USEDEF";valid_file "T";width
% 4.4555in;height 2.1378in;depth 0in;original-width 6.269in;original-height
% 2.9948in;cropleft "0";croptop "1";cropright "1";cropbottom "0";tempfilename
% 'PA5JSG02.wmf';tempfile-properties "XPR";}}
% \end{equation*}

Then by \cite[Theorem 6.47]{white84} we have a uniquely defined $2$-cell
embedding of $\Gamma $ into an orientable surface where the $2$-cells have
boundary labels $y^{n}$, with $n$ dividing $k$, or $(xy)^{m}$, with $m$
dividing $\ell $. Call such $2$-cells \emph{$y$-faces\/} or \emph{$xy$%
-faces\/} repectively.

The result is called a \emph{coset diagram\/}. The picture above shows a $%
\Gamma$ depicting the action of $\Delta(2,3,3)$ on a set of size six and the
resulting diagram embedded in the $2$-sphere.

From now we can omit the orientations on the $y$-edges of a coset diagram
with the understanding that the boundary edges of each $y$-face are oriented
anticlockwise around the face.

Suppose we have an action of some $\Delta $ on a set $S$ such that there are
two orbits of $\langle xy\rangle $, each consisting of $|S|/2$ elements. A 
\emph{januarial\/} $J$ is the resulting coset diagram. In particular, a
januarial is a diagram with two $xy$-faces. The examples in Figure 1 and 2\
are thus januarials. For a nice introduction to januarials, see \cite{jan},
from which the following concepts are taken.

Let $S_1$ and $S_2$ be the discs that result from taking the closures of the
two $xy$-faces.

It is convenient to collapse superfluous structure in a graph or diagram.
Let $J$ be a januarial arising from the coset graph $\Gamma $. Collapsing
all the $y$-faces to a point gives the \emph{companion graph\/} $\Gamma
^{\prime }$ and the \emph{companion diagram\/} $J^{\prime }$. Let $%
S_{1}^{\prime },S_{2}^{\prime }$ be the result of applying this to the discs 
$S_{1},S_{2}$.

\begin{figure}[h]
  \centering
\BoxedEPSF{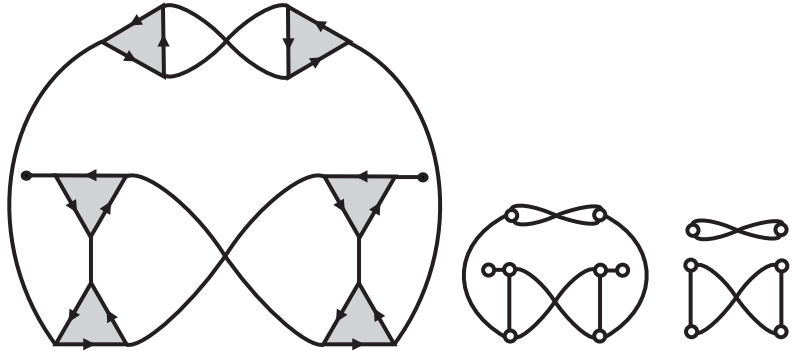 scaled 750}
  \caption{A graph of a
januarial $J$, its \emph{companion graph\/} $\Gamma ^{\prime }$ and the 
\emph{common graph\/} $\Upsilon$.}
\end{figure}
% \begin{equation*}
% \FRAME{itbpFU}{4.3785in}{1.9752in}{0in}{\Qcb{Figure 2: A graph of a
% januarial $J$, its \emph{companion graph\/} $\Gamma ^{\prime }$ and the 
% \emph{common graph\/} $\Upsilon .$}}{}{Figure}{\special{language "Scientific
% Word";type "GRAPHIC";maintain-aspect-ratio TRUE;display "USEDEF";valid_file
% "T";width 4.3785in;height 1.9752in;depth 0in;original-width
% 5.4872in;original-height 2.4613in;cropleft "0";croptop "1";cropright
% "1";cropbottom "0";tempfilename 'PA5KYR05.wmf';tempfile-properties "XPR";}}
% \end{equation*}

Let $g$ be the genus of the surface in which the januarial $J$ is embedded.
Let $R_{i}$ be a small closed neighbourhood of $S_{i}^{\prime }$ in $%
J^{\prime }$. Then $R_{i}$ is a closed surface with boundary. Let $g_{i}$ be
its genus and $h_{i}$ is the number of connected components in the boundary.

The \emph{common graph\/} $\Upsilon $ is the intersection $S_{1}^{\prime
}\cap S_{2}^{\prime }$. The following is in \cite[Lemma 3.5]{jan}:

\begin{lemma}
Let $P_{1}$ (resp. $P_{2}$) be the set of all paths that traverse successive
edges in $\Upsilon $ in the directions they are traversed by $S^{\prime
}_{1} $ (resp. $S^{\prime }_{2}$) in such a way that whenever such a path
reaches a vertex, it continues along the right-most of the other edges in $%
\Upsilon$ incident with that vertex. (The next edge is necessarily traversed
by $S_{1}$ (resp. $S_{2}$) in that direction.) All such paths close up into
circuits, and $P_{1}$ (resp. $P_{2}$) partitions $\Upsilon$, in the sense
that the union of the circuits is $\Upsilon $ and no two share an edge. The
cardinality of $P_{1}$ (resp. $P_{2}$) is $h_{2}$ (resp. $h_{1}$).
\end{lemma}

% In other words the restriction of $S_{1}$ (resp. $S_{2}$) to the subgraph $%
% \Upsilon $ gives the partition $P_{1}$ (resp. $P_{2}$) of $\Upsilon $ into $%
% h_{2}$ (resp. $h_{1}$) circuits.

A januarial is of \emph{simple type\/} if $\Upsilon$ is composed of $h$
disjoint simple circuits. In this case $h=h_1=h_2$ and we denote it by $%
(h,g_1,g_2)$. Otherwise, the januarial is of \emph{general type\/} $%
((h_{1},g_{1}),(h_{2},g_{2}))$.

The following are Lemmas 3.1, 3.4 and 3.6 of \cite{jan}:

\begin{lemma}
\label{1}Twice the genus $g$\ of a januarial equals the number $E$ of $x$%
-edges which are not loops minus the number $V$ of $y$-faces.
\end{lemma}

\begin{lemma}
\label{2}The genus of a januarial $J$ of simple type is $g=g_{1}+g_{2}+h-1$.
\end{lemma}

\begin{lemma}
\label{3}The $g_{i}$ satisfy $2-2g_{i}=V_{i}-E_{i}+h_{i}+1$ where $V_{i}$
and $E_{i}$ denote the number of vertices and edges in the subgraph of $%
\Gamma ^{\prime }$ visited by $S_{i}^{\prime }.$
\end{lemma}

In this paper we first generalize the formula in lemma \ref{2} for obtaining
the genus of a januarial of simple type to that of general type. Then we
address the questions posed by Graham Higman.

\begin{problem}
For a given $k$, what are the possible values for and interrelationships
between $h,g_{1}$ and $g_{2}$ for januarials of simple type? And a similar
question for januarial of general type.

Are there arbitrarily large values of $k$ for which there exist examples
with $h=1$?
\end{problem}

In \cite{qmsm} we describe a method that constructs a januarial from an
action of $\Delta _{k}$ on the projective line $PL(F_{p})$ over the field of
order $p$, with the property that $xy$ acts as permutation of $PL(F_{p})$ of
finite order $\ell$. Necessarily $2\ell =p+1$ for a januarial. It turns out
that there is a polynomial%
\begin{equation*}
f_{\ell }= \left\{%
\begin{array}{c}
{\displaystyle \binom{\ell -1}{0}\theta ^{(\ell -1)/2}-\binom{\ell -2}{1}%
\theta ^{(\ell -3)/2}+\binom{\ell -3}{2}\theta ^{(\ell -5)/2}-...\text{ for
odd }\ell} \\ 
\\ 
{\displaystyle \binom{\ell -1}{0}\theta ^{\ell /2-1}-\binom{\ell -2}{1}%
\theta ^{\ell /2-3}+\binom{\ell -3}{2}\theta ^{\ell /2-5}-...\text{ for even 
}\ell}%
\end{array}%
\right.
\end{equation*}
certain roots of which give linear fractional transformations of $PL(F_p)$.
For, let $\theta $ be a root of $f_{\ell }$ that is not a root of $f_{\ell
/s}$ where $s\mid \ell $. Let $X$, $Y$ be the linear fractional
transformations of $PL(F_{p})$ given by 
\begin{equation*}
X:z\mapsto \frac{az+cd}{cz-a} \text{ and } Y:z\mapsto \frac{ez+fd}{fz+b-e}
\end{equation*}
% \begin{eqnarray*}
% X &=&\frac{az+cd}{cz-a} \\
% Y &=&\frac{ez+fd}{fz+b-e}
% \end{eqnarray*}%
where $a,b,c,d,e,f\in 
%TCIMACRO{\U{2124} }%
%BeginExpansion
\mathbb{Z}
%EndExpansion
$, with $\nabla =-(a^{2}+dc^{2})\neq 0$, $r=a(2e-b)+2dcf$ and $%
1+df^{2}+e^{2}-eb=0$ and $\theta \nabla =r^{2}$.

Then the action of $\Delta _{k}$\ on the $p+1$ points of $PL(F_{p})$ gives a
januarial. Write $\bar{x}$ and $\bar{y}$ for the permutations induced by $X$
and $Y$. We say that this januarial has been constructed from the Hecke
group $\Delta_k$. 
% [Maybe expad this paragraph to includea summary of the method of
%   finding januarials from the action of Hecte groups on projective
%   lines over finite fields]

In \cite{jan} a technique is proposed for finding januarials. Associates are
constructed of known coset diagrams and investigated. Q. Mushtaq and S.
Mehwish \cite{qmsm} used the method of parametrization described in \cite%
{Qm, poly} to construct januarials from Hecke groups. In this paper we
construct januarials for the natural action of some permutation groups that
are homomorphic images of triangle groups.

\section{\label{6}Relationship between $h_{1},h_{2},g_{1}$ and $g_{2}$}

We now obtain the generalized formula for the genus of a januarial of
general type: %as was mentioned in the last paragraph of previous section.

\begin{lemma}
\label{4}Let $\alpha =v-e$ where $v$ is the number of vertices in $\Upsilon $
and $e$ is the number of edges. Then the genus of a general type januarial $%
((h_{1},g_{1}),(h_{2},g_{2}))$ is%
\begin{equation*}
g_{1}+g_{2}+(h_{1}+h_{2}+\alpha )/2-1.
\end{equation*}
\end{lemma}

\begin{proof}
By Lemma \ref{1} the genus $g$ of a januarial is $\ g=\frac{1}{2}(E-V)$ and
by Lemma \ref{3} $g_{i}=\frac{1}{2}(E_{i}-V_{i}-h_{i}+1)$. Since $%
E=E_{1}+E_{2}-e$ and $V=V_{1}+V_{2}-v,$ we get $g=\frac{1}{2}(E_{1}+E_{2}-e$ 
$-(V_{1}+V_{2}-v)).$ Simple calculations using the above lead to the
required result $g=g_{1}+g_{2}+\frac{1}{2}(h_{1}+h_{2}+\alpha )-1.$
\end{proof}

\vspace*{1em}

% Since a januarial of simple type, $\Upsilon $ is composed of disjoint simple 
% $h$ circuits one can easily see that $v=e,$ implying that $\alpha =0$. Also $%
% h=h_{1}=h_{2},$ imply that lemma \ref{2} can be deduced from the above
% result.

In \cite{qmsm} it is proved that the genus $g$ of januarials constructed
from Hecke Groups is given by 
\begin{equation*}
g=-\frac{p+1-\eta _{y}}{2k}+\frac{1}{4}(p+1-2\eta _{y}-\eta _{x}),
\end{equation*}
where $\eta _{x}$ and $\eta _{y}$ are the number of fixed points of $x$ and $%
y$ respectively. This equation implies that $k$ and $p$ are proportional.
Moreover, the number of fixed points of $x$ and $y$ are the same for fixed
values of $k$ and the prime $p$. Thus:

\begin{lemma}
\label{5} The genus of januarials constructed from Hecke groups for fixed
values of $k$ and the prime $p$, is a fixed value $g_{pk}$.
\end{lemma}

From lemmas \ref{2}, \ref{4} and \ref{5} we have:

\begin{proposition}
\label{p}When a januarial is constructed from a Hecke Group with $k$ and the
prime $p$ fixed, we have

\begin{enumerate}
\item $g_{1}+g_{2}+h=g_{pk}+1$ for januarials of simple type and,

\item $g_{1}+g_{2}+(h_{1}+h_{2}+\alpha )/2=g_{pk}+1$ for januarials of
general type.
\end{enumerate}
\end{proposition}

That is, for any $p$ and $k$ the sums $g_{1}+g_{2}+h$ and $%
g_{1}+g_{2}+(h_{1}+h_{2}+\alpha )/2$ are constant in these two cases. 
% of januarials of
% simple and general type.

\section{\label{7}Januarials with a simple disjoint circuit}

Analyzing januarials constructed from Hecke groups as described at the end
of Section \ref{januarials} it is observed that for higher values of $k$ the
chances of getting a januarial with simple disjoint circuits decreases.
Moreover the value of $p$ affects the number of circuits in the januarial.
Therefore finding a simple januarial with $h=1$ from a Hecke group becomes
almost impossible as we move to higher values of $k$. These remarks are made
precise in this section.

\begin{proposition}
\label{pro}The valency of each vertex of the subgraph $\Upsilon $ of a $k$%
-januarial is even.
\end{proposition}

\begin{proof}
Each vertex in the subgraph $\Upsilon $ is a $y$-face of the $k$-januarial
and the edges are $x$-monogons. Let $E$ be an $x$-monogon between the
vertices $U$ and $V$ of $\Upsilon .$

Since $\Upsilon =S_{1}^{\prime }\cap S_{2}^{\prime },$ the $x$-monogon $E$
lies both in the restriction of $S_{1}^{\prime }$ to $\Upsilon $ and in the
restriction of $S_{2}^{\prime }$ to $\Upsilon .$ If in the restriction of $%
S_{1}^{\prime }$ to $\Upsilon$ the monogon $E$ takes $U$ to $V$ then in the
restriction of $S_{2}^{\prime }$ to $\Upsilon$ we have that $E$ takes $V$ to 
$U.$ In other words $E$ is unidirectional in both the restrictions.

Now for the vertex $U$ in the restriction of $S_{1}^{\prime }$ to $\Upsilon$
the monogon $E$ is an outgoing edge for some circuit in $P_{1}.$ There must
then be a unidirectional incoming edge $E^{\prime }$ at $U$ in that circuit.

This implies that at any vertex of $\Upsilon$, incident edges exist
pairwise. So the valency of each vertex of the subgraph $\Upsilon $ of a $k$%
-januarial is even.
\end{proof}

\begin{theorem}
Every $3$-januarial is of simple type.
\end{theorem}

\begin{proof}
By Proposition \ref{pro}, at any vertex of the subgraph $\Upsilon $ of a $3$%
-januarial only two edges can be incident while moving along the orbit: one
pointing towards the vertex and the other pointing away from the vertex. So
only disjoint simple circuits will be formed in the subgraph $\Upsilon$,
giving it a simple $3$-januarial.
\end{proof}

\begin{figure}
  \centering
\BoxedEPSF{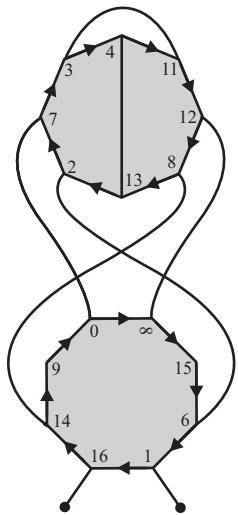 scaled 500}
  \caption{A diagram for $\Delta(17,17,8)$. Note that there is an
    $x$-edge connecting the vertices $4$ and $13$.} 
\label{figure:D(17,17,8)}
\end{figure}

\vspace*{1em}

When $k\geq 4$ the likelihood of getting simple $k$-januarials decreases
expeditiously. For an example we construct an $8$-januarial from the Hecke
group $\Delta _{8}$ acting on $PL(F_{17})$. We have $\ell =9$ and hence%
\begin{equation*}
f_{9}=1\theta ^{4}-8\theta ^{3}+21\theta ^{2}-20\theta +5.
\end{equation*}%
The roots of $f_{9}$ yielding januarials are $9,15$ and $16$. For an $8$%
-januarial, if we consider the root $\theta =9$ we get the linear fractional
transformations 
\begin{equation*}
X:z\mapsto \frac{z+10}{10z-1}\text{ and }Y:z\mapsto \frac{4}{4z+8}
\end{equation*}%
%
% \[
% X:z\rightarrow \frac{z+10}{10z-1},\ \ \ \ \ \ \ \acute{y}%
% :z\rightarrow \frac{4}{4z+8}.
% \]%
The action of $X$ and $Y$ on $PL(F_{17})$ gives 
\begin{eqnarray*}
\bar{x} &=&(0,7)(1,5)(2,6)(3,11)(4,13)(8,14)(9)(10,16)(12,\infty )(15) \\
\bar{y} &=&(0,9,14,16,1,6,15,\infty )(2,13,8,12,11,4,3,7)(5)(10) \\
\bar{x}\bar{y} &=&(0,4,7,19,13,1,17,5,22,11,14,18)(\infty
,20,6,16,8,3,9,15,10,2,12,21).
\end{eqnarray*}%
Hence the coset diagram $D(17,17,8)$ is given by Figure
\ref{figure:D(17,17,8)}. The two partitions of $\Upsilon $ into
circuits is drawn in Figure \ref{figure:D(17,17,8)_common}. 

% \begin{equation*}
% \FRAME{itbpF}{1.3923in}{3.045in}{0in}{}{}{Figure}{\special{language
% "Scientific Word";type "GRAPHIC";maintain-aspect-ratio TRUE;display
% "USEDEF";valid_file "T";width 1.3923in;height 3.045in;depth
% 0in;original-width 2.4197in;original-height 5.335in;cropleft "0";croptop
% "1";cropright "1";cropbottom "0";tempfilename
% 'PA5IUT00.wmf';tempfile-properties "XPR";}}
% \end{equation*}%
%
%
% \[
% \FRAME{itbpF}{1.1597in}{2.412in}{0in}{}{}{Figure}{\special{language
% "Scientific Word";type "GRAPHIC";maintain-aspect-ratio TRUE;display
% "USEDEF";valid_file "T";width 1.1597in;height 2.412in;depth
% 0in;original-width 1.8101in;original-height 3.7922in;cropleft "0";croptop
% "1";cropright "1";cropbottom "0";tempfilename
% 'P2HH5H03.wmf';tempfile-properties "XPR";}}
% \]%
%
%
% \[
% \FRAME{itbpF}{2.6394in}{2.2623in}{0in}{}{}{Figure}{\special{language
% "Scientific Word";type "GRAPHIC";maintain-aspect-ratio TRUE;display
% "USEDEF";valid_file "T";width 2.6394in;height 2.2623in;depth
% 0in;original-width 4.4417in;original-height 3.8035in;cropleft "0";croptop
% "1";cropright "1";cropbottom "0";tempfilename
% 'P2HH3J02.wmf';tempfile-properties "XPR";}}
% \]%

\begin{figure}
  \centering
\BoxedEPSF{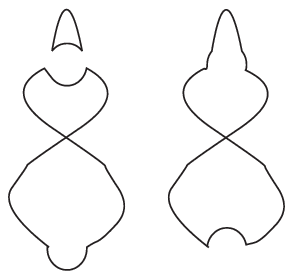 scaled 750}
  \caption{}
\label{figure:D(17,17,8)_common}
\end{figure}

% \begin{equation*}
% \FRAME{itbpF}{1.1294in}{1.081in}{0in}{}{}{Figure}{\special{language
% "Scientific Word";type "GRAPHIC";maintain-aspect-ratio TRUE;display
% "USEDEF";valid_file "T";width 1.1294in;height 1.081in;depth
% 0in;original-width 2.6913in;original-height 2.5737in;cropleft "0";croptop
% "1";cropright "1";cropbottom "0";tempfilename
% 'P9L3VM07.wmf';tempfile-properties "XPR";}}
% \end{equation*}%
%
% \[
% \FRAME{itbpF}{1.1294in}{1.081in}{0in}{}{}{Figure}{\special{language
% "Scientific Word";type "GRAPHIC";maintain-aspect-ratio TRUE;display
% "USEDEF";valid_file "T";width 1.1294in;height 1.081in;depth
% 0in;original-width 2.6913in;original-height 2.5737in;cropleft "0";croptop
% "1";cropright "1";cropbottom "0";tempfilename
% 'P2HHWN04.wmf';tempfile-properties "XPR";}} 
% \]
This is an 8-januarial of general type with $h_{1}=2$ and $h_{2}=1.$ It has $%
8$ edges which are not loops and $4$ $y$-faces, giving genus $g=\frac{1}{2}%
(8-4)=2.$ Calculating the genera using $2-2g_{i}=V_{i}-E_{i}+h_{i}+1\,$\
gives $g_{1}=\frac{1}{2}(1-2+5-2)=1$ and $g_{2}=\frac{1}{2}(1-4+6-1)=1.$
Alternatively by Lemma 5 the genus of this general type januarial $%
((2,1),(1,1))$ is $g=1+1+(2+1-1)/2-1=2.$

\begin{figure}
  \centering
\BoxedEPSF{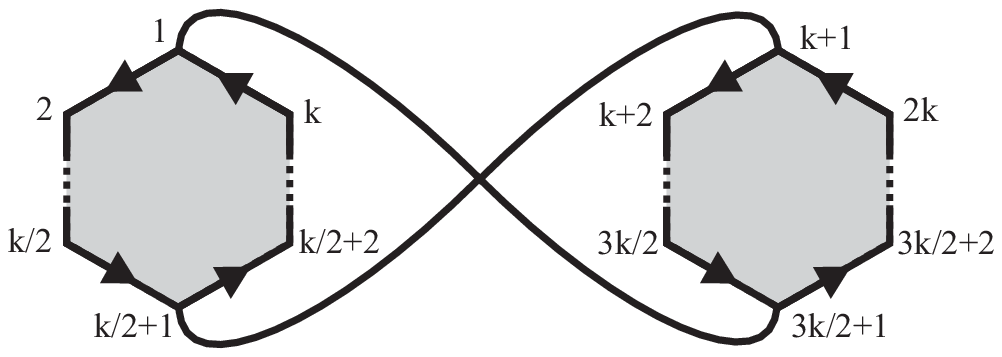 scaled 750}
  \caption{}
\label{figure:even_simple}
\end{figure}

From the above example it follows that the action of Hecke groups $\Delta_k$
for larger $k$'s on smaller prime fields do not yield simple januarials.
This leads to the question of whether there exist examples of simple
januarials for larger values of $k$. If so, then this answers Higman's
question of whether there are arbitrarily large values of $k$ for which
there exist examples with $h=1$:

\begin{theorem}
\label{t}There exist simple januarials with $h=1$ for all values of $k$.
\end{theorem}

By taking into consideration the definition of januarials and our
requirement that we have one simple disjoint circuit in the common graph, in
the proof below we construct the $k$- januarials for even and odd $k$
separately.

\begin{proof}
Consider the diagram in Figure \ref{figure:even_simple}. It is easy to see it as a permutation action of
the subgroup generated by 
\begin{equation*}
\bar{x}=\left( 1,3k/2+1\right) \left( k/2+1,k+1\right) \text{ and }\bar{y}%
=\left( 1,2,\ldots ,k\right) \left( k+1,k+2,\ldots ,2k\right) 
\end{equation*}%
of $S_{2k}.$%

% \begin{equation*}
% \FRAME{itbpF}{3.2949in}{1.164in}{0in}{}{}{Figure}{\special{language
% "Scientific Word";type "GRAPHIC";maintain-aspect-ratio TRUE;display
% "USEDEF";valid_file "T";width 3.2949in;height 1.164in;depth
% 0in;original-width 3.9747in;original-height 1.3854in;cropleft "0";croptop
% "1";cropright "1";cropbottom "0";tempfilename
% 'PA7FHZ07.wmf';tempfile-properties "XPR";}}
% \end{equation*}%
%
%
% \[
% \FRAME{itbpF}{2.8193in}{1.1986in}{0in}{}{}{Figure}{\special{language
% "Scientific Word";type "GRAPHIC";maintain-aspect-ratio TRUE;display
% "USEDEF";valid_file "T";width 2.8193in;height 1.1986in;depth
% 0in;original-width 3.8121in;original-height 1.6042in;cropleft "0";croptop
% "1";cropright "1";cropbottom "0";tempfilename
% 'P2J7RA05.wmf';tempfile-properties "XPR";}}
% \]%
This is a $k$-januarial for even values of $k$. The companion graph and the
two partitions of the subgraph $\Upsilon$ are show in Figure \ref{figure:even_simple_common}.

\begin{figure}[h]
  \centering
\BoxedEPSF{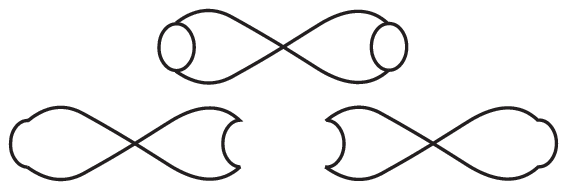 scaled 750}
  \caption{}
\label{figure:even_simple_common}
\end{figure}

% \begin{equation*}
% \FRAME{itbpF}{2.22in}{0.7213in}{0in}{}{}{Figure}{\special{language
% "Scientific Word";type "GRAPHIC";maintain-aspect-ratio TRUE;display
% "USEDEF";valid_file "T";width 2.22in;height 0.7213in;depth
% 0in;original-width 5.0505in;original-height 1.6207in;cropleft "0";croptop
% "1";cropright "1";cropbottom "0";tempfilename
% 'P9L3VM09.wmf';tempfile-properties "XPR";}}
% \end{equation*}%
%
% \[
% \FRAME{itbpF}{2.22in}{0.7213in}{0in}{}{}{Figure}{\special{language
% "Scientific Word";type "GRAPHIC";maintain-aspect-ratio TRUE;display
% "USEDEF";valid_file "T";width 2.22in;height 0.7213in;depth
% 0in;original-width 5.0505in;original-height 1.6207in;cropleft "0";croptop
% "1";cropright "1";cropbottom "0";tempfilename
% 'P2J8DU06.wmf';tempfile-properties "XPR";}}
% \]%
This shows that the januarial is of simple type with $h=1$. For odd values of $k$
consider the diagram in Figure \ref{figure:odd_simple} which shows a permutation action of $\Delta
(2,k,2k)$ on 4k points.

\begin{figure}[h]
  \centering
\BoxedEPSF{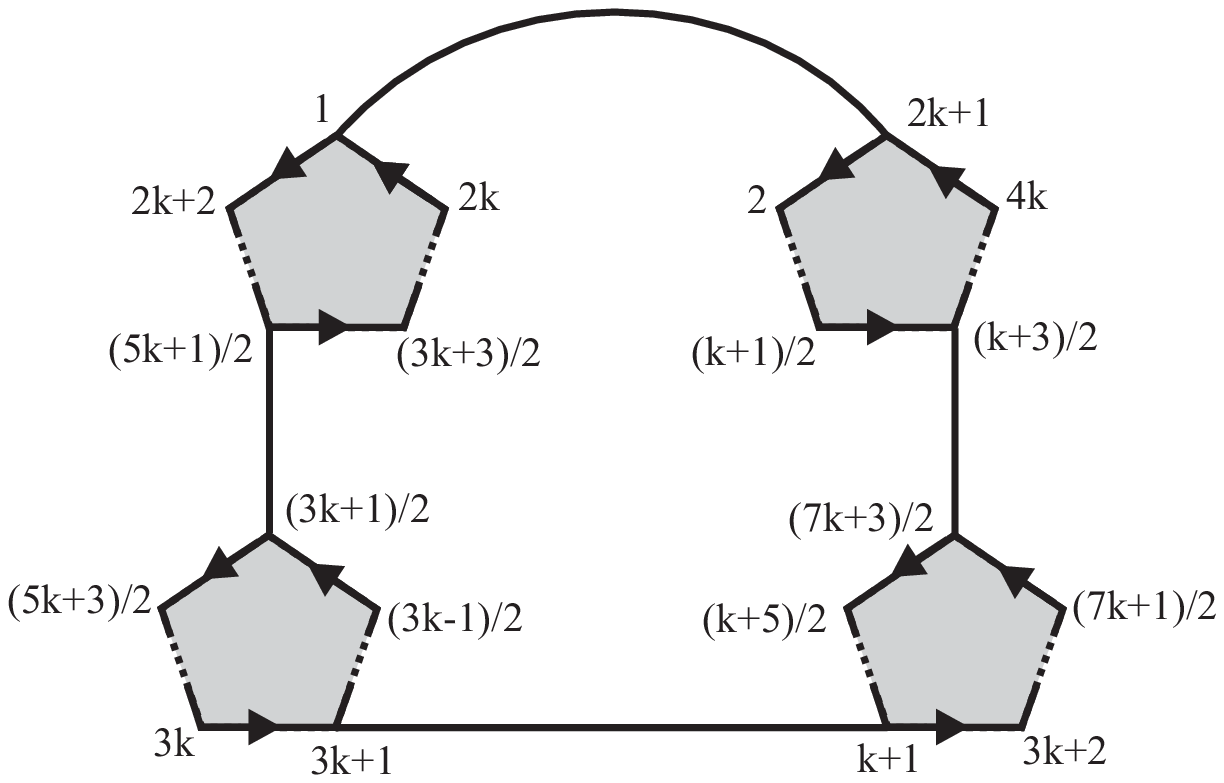 scaled 750}
  \caption{}
\label{figure:odd_simple}
\end{figure}

% \begin{equation*}
% \FRAME{itbpF}{3.672in}{2.3402in}{0in}{}{}{Figure}{\special{language
% "Scientific Word";type "GRAPHIC";maintain-aspect-ratio TRUE;display
% "USEDEF";valid_file "T";width 3.672in;height 2.3402in;depth
% 0in;original-width 4.868in;original-height 3.0926in;cropleft "0";croptop
% "1";cropright "1";cropbottom "0";tempfilename
% 'PA7D4O06.wmf';tempfile-properties "XPR";}}
% \end{equation*}%
%
% \[
% \FRAME{itbpF}{3.6668in}{2.0669in}{0in}{}{}{nq8bv103.bmp}{\special{language
% "Scientific Word";type "GRAPHIC";maintain-aspect-ratio TRUE;display
% "USEDEF";valid_file "F";width 3.6668in;height 2.0669in;depth
% 0in;original-width 14.2288in;original-height 8.0004in;cropleft "0";croptop
% "1";cropright "1";cropbottom "0";filename 'NQ8BV103.bmp';file-properties
% "XNPEU";}}
% \]%
It is a $k$-januarial and the two faces of $xy$ are labeled by $1$ $2$ $%
3\ldots \ 2k$ and $2k+1$ $2k+2$ $2k+3\ldots \ 4k$. The companion graph and
the two partitions of the subgraph $\Upsilon $ are shown in Figure \ref{figure:odd_simple_common}.

\begin{figure}[h]
  \centering
\BoxedEPSF{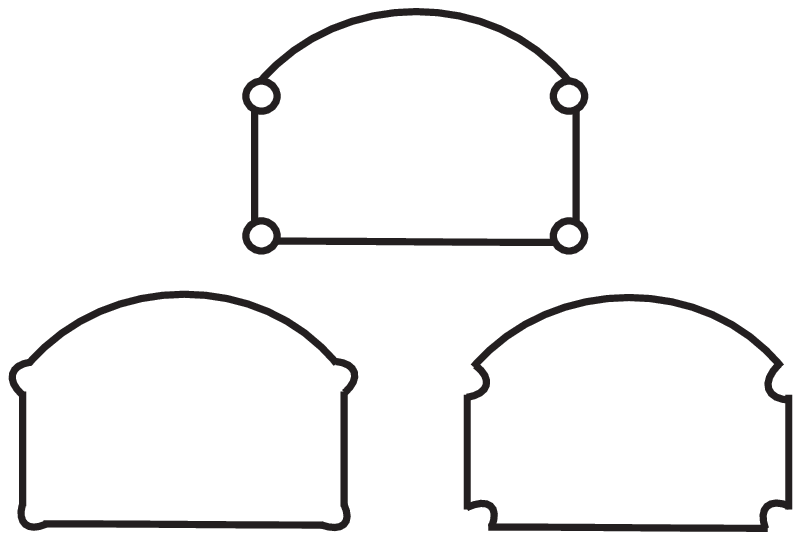 scaled 750}
  \caption{}
\label{figure:odd_simple_common}
\end{figure}

% \begin{equation*}
% \FRAME{itbpF}{1.7184in}{1.1597in}{0in}{}{}{Figure}{\special{language
% "Scientific Word";type "GRAPHIC";maintain-aspect-ratio TRUE;display
% "USEDEF";valid_file "T";width 1.7184in;height 1.1597in;depth
% 0in;original-width 3.1644in;original-height 2.124in;cropleft "0";croptop
% "1";cropright "1";cropbottom "0";tempfilename
% 'PA7GB809.wmf';tempfile-properties "XPR";}}
% \end{equation*}%
%
%
% \[
% \FRAME{itbpF}{1.8127in}{1.094in}{0in}{}{}{Figure}{\special{language
% "Scientific Word";type "GRAPHIC";display "USEDEF";valid_file "T";width
% 1.8127in;height 1.094in;depth 0in;original-width 2.5936in;original-height
% 1.8559in;cropleft "0";croptop "1";cropright "1";cropbottom "0";tempfilename
% 'P1T8AO02.wmf';tempfile-properties "XPR";}}
% \]%
This shows that the januarial is of simple type with $h=1$.
\end{proof}

\vspace*{1em}

These januarials are the result of actions of $\Delta (2,k,\ell )$ on the
set of $2k$ and $4k$ points for even and odd values of $k$ respectively,
which can easily be extended to larger finite sets in many ways. For
example, we can insert $k$-polygons in the diagrams above while maintaining
the symmetry. This technique is also used by Q. Mushtaq and F. Shaheen in 
\cite{qmf}. Thus, there exist many more januarials with the required
condition.

\end{document}